\documentclass[12pt]{amsart}

\usepackage{amssymb}

\usepackage{mathrsfs}
\usepackage{xcolor}
\usepackage{enumitem}

\usepackage{graphicx}

\makeatletter
\@namedef{subjclassname@2020}{%
  \textup{2020} Mathematics Subject Classification}
\makeatother

\usepackage[T1]{fontenc}
\newtheorem{theorem}{Theorem}[section]
\newtheorem{corollary}[theorem]{Corollary}
\newtheorem{lemma}[theorem]{Lemma}
\newtheorem{proposition}[theorem]{Proposition}

\theoremstyle{definition}
\newtheorem{definition}[theorem]{Definition}

\newtheorem{example}[theorem]{Example}

\numberwithin{equation}{section}

\newcommand{\X}{\mathscr{X}}

\newcommand{\D}{\mathbb{D}}
\newcommand{\SA}{\mathscr{S\mspace{-5mu}A}}
\newcommand{\SApO}{\mathscr{S\mspace{-5mu}A}_\psi (\Omega)}

\newcommand{\CB}{\mathcal{C}_b(\Psi)}

\begin{document}

%%%%% To ease editing, for IMPAN journals add:

\baselineskip=17pt

%%%%%%%%%%%%%%%%

\title{Schur-Agler class and Carath\'eodory extremal functions}

\author[A. Biswas]{Anindya Biswas}
\address{Department of Mathematics and Statistics, Masaryk University, Brno}
\email{anindyab132652@gmail.com}
\address{ORCID: 0000-0002-7805-9446}
\thanks{}

\date{}

\begin{abstract}
	
	We study the role of Carathéodory extremal functions in the Schur–Agler class $\SA_\psi(\Omega)$ associated with a family \(\Psi\) of test functions on a Carath\'eodory hyperbolic domain $\Omega\subset\mathbb{C}^m$.
	The paper makes three main contributions. 
	First, we introduce a function $d_\Psi$ using the Schur-Agler class $\SA_\Psi(\Omega)$ analogous to the definition of Carath\'eodory pseudodistance in terms of the Schur class and show that $d_\Psi$ can be computed using only the test functions. Second, we prove that every Carathéodory hyperbolic domain admits a family $\Psi$ of test functions such that $\SA_\Psi(\Omega)$ coincides with Schur class. Third, we prove that if a domain possesses a finite family of test functions that generates the entire Schur class, then the domain must be biholomorphic either to the open unit disc $\mathbb{D}$ or to the bidisc $\mathbb{D}^2$. As applications we characterise the possible values of Carathéodory extremals on the Drury–Arveson space and we obtain an operator‑theoretic Herglotz representation formula.
\end{abstract}



\maketitle

\section{Introduction.}

Suppose that $\Omega$ is a domain in a finite dimensional complex space $\mathbb{C}^m$ such that the bounded holomorphic functions separate the points of $\Omega$. Such a domain is called Carath\'eodory hyperbolic. For the well-studied domains the open unit disc $\D$ and the bidisc $\D^2$, any element of the Schur class (the closed unit ball of the Banach algebra of all bounded holomorphic functions on $\Omega$, denoted by $\mathscr{S}(\Omega)$) has a structure known as the \textit{realization formula}, which states that given any $f\in \mathscr{S}(\Omega)$, there is a Hilbert space $\X$, a unitary operator $U:\mathbb{C}\oplus \X\rightarrow \mathbb{C}\oplus \X$ and an operator $z\in B(\X)$ which is identified with the identity function on $\Omega$ such that with $$ U =
\bordermatrix{ & \mathbb{C} & \X \cr
	\mathbb{C} & a & B \cr
	\X & C & D},
$$ one can express $f$ as $f(z)=a+Bz(I_{\X}-Dz)^{-1}C$ for all $z\in \Omega$.

In the theory of interpolation, the realization formula has certain important implications. For instance, the realization formula leads to the fact that every solvable Pick interpolation problem on $\D^2$ has a rational inner solution (Corollary 2.13 in \cite{A-M-Paper}, Corollary 11.54 in \cite{A-M}). Due to such importance, the property of having a realization formula has been investigated further and it has been found that in the case of the annulus (Theorem 4.4 in \cite{B_H}, Proposition 5.2 in \cite{D-M}) and the symmetrized bidisc (page 510 in \cite{B-S}) similar form can be established, though not as convenient as in the case of $\D$ and $\D^2$, but still provide useful and far reaching consequences (\cite{BhatBisCha}). 

In the case of $\D$ and $\D^2$, in a certain sense, only a finite number of elementary functions can generate the entire Schur class, namely, the identity function on $\D$ and the coordinate functions on $\D^2$, although this is not the case for the other domains. For the annulus and the symmetrized bidisc, one requires infinitely many functions to provide a realization formula, and in this article we will see that beyond the domains $\D$ and $\D^2$, it is indeed necessary to consider such an infinite collection.

Generalizing the idea that a certain collection of functions has the property of generating the whole Schur class leads us to a notion known as the \textit{Schur-Agler class}, or more precisely, the \textit{$\Psi$-Schur-Agler class} for a given collection of functions $\Psi$ (\cite{B_H}, \cite{BhatBisCha}, \cite{D-M}). We call such a collection $\Psi$ a set of test functions and they are introduced in the following way.
\begin{definition}
	A collection $\Psi$ of $\mathbb{C}$-valued functions on $\Omega$ is called a set of \textit{test functions} if the following conditions hold.
	\begin{enumerate}
		\item $\sup\{|\psi(z)|:\psi\in \Psi\}<1$, for each $z\in \Omega$.
		\item For each finite subset $F$ of $\Omega$, the collection $\{\psi |_F:\psi\in \Psi\}$, together with the constant function, generates the algebra of all $\mathbb{C}$-valued functions on $F$.
	\end{enumerate}
\end{definition}

We can topologize $\Psi$ by considering it as a subspace of $\overline{\mathbb{D}}^\Omega$ equipped with the product topology. For every $z\in \Omega$, we define an element $E(z)$ in $\mathcal{C}_b(\Psi)$, {\em the $C^*$-algebra of all bounded continuous functions on} $\Psi$, such that $E(z)(\psi)=\psi(z)$. Clearly,
\begin{align}\label{NormOfEz}
	\| E(z)\|= \text{sup}_{\psi\in \Psi} |\psi(z)|<1
\end{align}
for each $z\in \Omega$. These functions $E(z)$ play a key role in the development of the theory of $\Psi$-Schur-Agler class.

Given such a set $\Psi$, we next introduce a collection of positive semidefinite functions or kernels $k:\Omega\times \Omega\rightarrow \mathbb{C}$ satisfying the property
\begin{align*}
	\big((1-\psi(z)\overline{\psi(w)})k(z,w)\big)_{z,w\in \Omega}\geq \mathbf{0},\,\ \text{for all}\,\, \psi\in \Psi.
\end{align*}
These functions $k$ are called the $\Psi$-admissible kernels and we will denote the set of all $\Psi$-admissible kernels by $\mathcal{K}_\Psi$.

Now we can describe the $\Psi$-Schur-Agler class. 
\begin{definition}
	Given a collection of test functions $\Psi$ on a domain $\Omega$, the $\Psi$-Schur-Agler class is the collection of functions $f:\Omega\rightarrow \mathbb{C}$ for which
	\begin{align*}
		\big((1-f(z)\overline{f(w)})k(z,w)\big)_{z,w\in \Omega}\geq \mathbf{0},\,\ \text{for all}\,\, k\in\mathcal{K}_\Psi.
	\end{align*}
\end{definition}
The $\Psi$-Schur-Agler class on $\Omega$ will be denoted by $\SA_\psi(\Omega)$, or $\SA_\Psi$ when the underlying domain is clear from the context.

One should note that, we do not require $\Psi$ to be a subset of $\mathscr{S}(\Omega)$, but then $\SA_\psi(\Omega)$, which also contains $\Psi$, is not a subset of $\mathscr{S}(\Omega)$. On the other hand, if we take $\Psi$ from $\mathscr{S}(\Omega)$, $\SA_\psi(\Omega)$ is then contained in $\mathscr{S}(\Omega)$ and it inherits a number of important properties like power series expansion and Heine–Borel property from $\mathscr{S}(\Omega)$. Also, we observe that, since to check positivity of $\big((1-f(z)\overline{f(w)})k(z,w)\big)_{z,w\in \Omega}$ one needs to consider only finitely many points from $\Omega$ at a time, $\SA_\psi(\Omega)$ is closed under pointwise convergence. For the development with holomorphic test functions, an interested reader may consult \cite{BhatBisCha}. The main result providing the structure of $\SA_\psi(\Omega)$ is the following (see Theorem 1.3 in \cite{B_H}, Theorem 2.1 in \cite{BhatBisCha}, Theorem 2.2 in \cite{D-M}).
\begin{theorem}\label{CharacterSchurAgler}
	For a function $f:\Omega \rightarrow \overline{\D}$, the following statements are equivalent.
	\begin{enumerate}
		\item $f\in \SApO$.
		\item $f$ has an Agler decomposition, that is, there exists a completely
		positive kernel $\Gamma:\Omega\times\Omega\rightarrow\CB^*$ such that 
		\begin{align}\label{AlterAglerDecomp}
			1-f(z)\overline{f(w)}=\Gamma(z,w)(1-E(z)E(w)^*)\,\,\,\text{for all $z,w\in \Omega$.}
		\end{align}
		\item\label{Realization} There is a Hilbert space $\X$, a unitary $U:\mathbb{C}\oplus \X\rightarrow \mathbb{C}\oplus \X$ and a unital $*$-representation $\rho:\CB\rightarrow B(\X)$ such that writing $U$ as 
		\begin{align*}
			U =
			\bordermatrix{ & \mathbb{C} & \X \cr
				\mathbb{C} & a & B \cr
				\X & C & D},
		\end{align*}
		one can express $f$ as $f(z)=a+B\rho(E(z))(I_{\X}-D\rho(E(z)))^{-1}C$ for all $z\in \Omega$.
		\item There is a Hilbert space $\X$, a unital $*$-representation $\rho:\CB\rightarrow B(\X)$ and a map $h:\Omega\rightarrow \X$ such that 
		\begin{align}\label{AglerDecompInner}
			1-\overline{f(z)}f(w)=\langle \rho(1-E(z)^* E(w))h(w),h(z)\rangle_\X \,\,\,\text{for all $z,w\in \Omega$.}
		\end{align}
		
	\end{enumerate}
\end{theorem}

By the completely positive kernel $\Gamma$ in (\ref{AlterAglerDecomp}) we mean that for any $z_1,\ldots,z_n\in \Omega$, $a_1,\ldots,a_n\in\CB$, the matrix $\big(\Gamma(z_i,z_j)(a^* _i a_j)\big)_{1\leq i,j\leq n}$ is positive semidefinite. It is known (\cite{B-B-F-t}, \cite{B_H}, \cite{B-B-L-S}) that, for any such $\Gamma$, there is a Hilbert space $\X$, a unital $*$-representation $\rho:\CB\rightarrow B(\X)$ and a map $g:\Omega\rightarrow \X (=B(\X,\mathbb{C}))$ such that
\begin{align}\label{Kolmogorov Decomp}
	\Gamma(z,w)(\delta)=\langle \rho(\delta)g(w)^*,g(z)^*\rangle.
\end{align}
%Define $\overline{\mathbb{D}}_X$

In this article, we will always assume that our collection $\Psi$ of test functions is a subset of $\mathscr{S}(\Omega)$ and that there is a point $w_0\in \Omega$ such that $\psi(w_0)=0$ for all $\psi\in \Psi$. As a consequence, $\SA_\Psi\subset\mathscr{S}(\Omega)$ and $E(w_0)=0$. It is known that this assumption does not alter the $\SA_\Psi$ (see Section 3 in \cite{BhatBisCha}).

We now recall the Carath\'eodory pseudodistance on $\Omega$. Given two points $z_1,z_2\in \Omega$, it is defined by 
\begin{align}\label{CaraDefn}
	c_\Omega ^*(z_1,z_2)=sup\Big\{m(f(z_1),f(z_2))=\Big|\frac{f(z_1) -f(z_2)}{1- f(z_1) \overline{f(z_2)}}\Big|:f\in \mathscr{S}(\Omega), f(\Omega)\subset \D\Big\}.
\end{align}
It is known that the Carath\'eodory pseudodistance is always attained and the functions that attain it are called \textit{Carath\'eodory extremals}. If a collection $\mathcal{C}$ of functions in $\mathscr{S}(\Omega)$ has the property that the Carath\'eodory pseudodistance between any two points in $\Omega$ is attained by some function in $\mathcal{C}$, we call this collection  a \textit{universal set} for the \textit{Carath\'eodory extremal problem}. It turns out, if one imposes a minimality condition on $\mathcal{C}$, then in some cases the collection is uniquely determined. Here are a few examples of domains with a minimal $\mathcal{C}$ (Section 5 in \cite{AglLykYouCharac}).

\begin{example}
	\begin{enumerate}
		\item For the open unit disc $\D$, the Carath\'eodory extremal is the identity function on $\D$ and it is unique upto composition with automorphisms of $\D$.
		\item In the case of bidisc, the extremals are  the coordinate functions and again they are unique (Lemma \ref{Uniqueness_Cara_Ext_Poly} in Section \ref{FinitenessOf}). 
		\item On the symmetrized bidisc $\mathbb{G}=\{(z_1+z_2,z_1z_2):z_1,z_2\in \D\}$, the collection is given by (Theorem 1.1 in \cite{AglerYounHyperbolic}, Theorem Theorem 7.1.1 in \cite{J-P-Invariant})
		\begin{align*}
			\{\mathbb{G}\ni(s,p)\mapsto \frac{2zp-s}{2-zs}\in \D:z\in \overline{\D}\}.
		\end{align*}
	\end{enumerate}
\end{example}
It turns out, the collections we described above serve as test functions that generate the entire Schur class (Theorem 11.13 in \cite{A-M}, page 510 in \cite{B-S}). Therefore, it is natural to ask whether the Carath\'eodory extremals have any role in the development of the Schur-Agler class.

Analogous to the definition of $c^* _\Omega$, we now consider the following: Suppose that $z_1, z_2\in \Omega$ and define
\begin{align}\label{invariant_distance}
	d_\Psi (z_1,z_2)=\sup \Big\{ m(f(z_1), f(z_2))=\Big|\frac{f(z_1) -f(z_2)}{1- f(z_1) \overline{f(z_2)}}\Big|:f\in \SA_\Psi\Big\}.
\end{align}
Since $\SA_\Psi$ is contained in the Schur class, using Montel's theorem, we can say that there is an $f\in \SA_\Psi$ such that $d_\Psi (z_1,z_2)=m(f(z_1), f(z_2))$.

Our first two results show that $d_\Psi$ has an expression in terms of the admissible kernels and is attained by the test functions in $\Psi$. So if we want to have $\SA_\Psi =\mathscr{S}(\Omega)$, then $\Psi$ needs to be a universal set for the Carath\'eodory extremal problem in $\Omega$.

When the domain $\Omega$ is the polydisc $\D^m$, we have proved that a minimal universal set contains exactly $m$ functions (upto composition with automorphisms of $\D$), namely, the coordinate functions. But this does not hold in general, that is, a minimal universal set for the Carath\'eodory extremal problem is generically infinite. Thus, the problem of generating the Schur class with finitely many $\mathbb{C}$-valued test functions has very little hope. However, the importance of the realization formula asks for further development and we have provided, for any Carath\'eodory hyperbolic domain $\Omega$, a collection $\Psi$ of test functions such that $\SA_\Psi=\mathscr{S}(\Omega)$.

It is known that upto some extent the Carath\'eodory extremals carry the data sufficient to characterize the underlying domain (\cite{AglLykYouCharac}, Theorem 4.4 in \cite{KosinZwoExt}) and if $\mathscr{S}(\Omega)$ can be generated by a finite collection of test functions, then we may expect to recover the domain from the properties of the test functions. Our main result in Section \ref{FinitenessOf} proves exactly this. We show that under certain natural conditions, if $\SA_\Psi=\mathscr{S}(\Omega)$ holds with a finite collection $\Psi$, then $\Omega$ is necessarily biholomorphic to either $\D$ or $\D^2$.

Lastly, we give two applications of our results. The first one, comes from the observation that the reproducing kernel of the Drury-Arveson space is closely related to the Carath\'eodory pseudodistance on the Euclidean ball, describes the range of the Carath\'eodory extremals in the closed unit ball of the multiplier algebra of the Drury-Arveson space. The second application utilizes the test function we have introduced and gives an operator-theoretic Herglotz representation for any Carath\'eodory hyperbolic domain.
\section{Extremal functions and $\Psi$}\label{FinitenessOf}
We start with a description of $d_\Psi$ in terms of he admissible kernel.

\begin{theorem}
	\begin{align}\label{d_in_terms_of_adm_kernels}
		d_\Psi (z_1,z_2)=\inf& \Big\{\sqrt{1-\frac{|k(z_1, z_2)|^2}{k(z_1,z_1)k(z_2,z_2)}}: k(z_1,z_1)\neq 0\neq k(z_2,z_2),\\
		& k\,\text{is an admissible kernel, i.e., }k\in \mathcal{K}_\Psi, \Big\}.\notag
	\end{align}
\end{theorem}

\begin{proof}
	We have that there is an $f\in \SA_\Psi$ such that $$d_\Psi (z_1,z_2)=m(f(z_1), f(z_2)).$$
	So, for any $k\in \mathcal{K}_\Psi$ we have
	\begin{align*}\begin{pmatrix}
			(1-|f(z_1)|^2)k(z_1,z_1)& (1- f(z_1) \overline{f(z_2)})k(z_1,z_2)\\
			(1- f(z_2) \overline{f(z_1)})k(z_2,z_1)& 	(1-|f(z_2)|^2)k(z_2,z_2)
		\end{pmatrix} \geq \mathbf{0}.
	\end{align*} 
	
	Using the fact that the above matrix has non-negative determinant, we get 
	\begin{align*}
		\sqrt{1-\frac{|k(z_1, z_2)|^2}{k(z_1,z_1)k(z_2,z_2)}}\geq m(f(z_1), f(z_2))=d_\Psi (z_1,z_2).
	\end{align*}
	Now, we consider the following function on $\Omega\times\Omega$
	\begin{align}\label{kernel_associated_to_function}
		\tilde{k}(x,y)&=\frac{1}{1- f(x)\overline{f(y)}},\,x,y\in\{z_1,z_2\}\\
		&=0,\,\,\text{otherwise}.\notag
	\end{align}
	
	\textbf{Claim.} $\tilde{k}$ is an admissible kernel.
	
	\textit{Proof of the claim.} We will show that $\big((1-\psi(z)\overline{\psi(w)})\tilde{k}(z,w)\big)_{z,w\in \Omega}$ is positive semidefinite for all $\psi\in\Psi$. Let $F$ be any finite subset of $\Omega$. If none of $z_1$ and $z_2$ belong to $F$ then $\big((1-\psi(z)\overline{\psi(w)})\tilde{k}(z,w)\big)_{z,w\in F}$ is a matrix with all entries equal to zero and therefore it is trivially positive semidefinite. If $z_1\in F$ but $z_2\notin F$ then the matrix has only one nonzero entry $\frac{1-|\psi(z_1)|^2}{1-|f(z_1)|^2}$ which occurs on the diagonal and hence the matrix is positive semidefinite. Now suppose that $z_1,z_2\in F$. We may write $F=\{z_1,z_2,z_3\ldots,z_l\}$. So the only nontrivial entries are in the leading principal submatrix of order 2 which is given by $\begin{pmatrix}
		\frac{1-\psi(z_i)\overline{\psi(z_j)}}{1-f(z_i)\overline{fi(z_j)}}
	\end{pmatrix}_{1\leq i,j\leq 2}.$ Since $\Psi\subset \SA_\Psi$ and $m(f(z_1),f(z_2))\geq m(h(z_1),h(z_2))$ for all $h \in \SA_\Psi$, the matrix is positive semidefinite and the claim is proved.

	Now it is easy to see that
	\begin{align*}
		\sqrt{1-\frac{|\tilde{k}(z_1, z_2)|^2}{\tilde{k}(z_1,z_1)\tilde{k}(z_2,z_2)}}=m(f(z_1), f(z_2))=d_\Psi (z_1,z_2)
	\end{align*}
	and the proof of the theorem is complete.
\end{proof}

\begin{corollary}
	\begin{align}\label{d_in_Psi}
		d_\Psi (z_1,z_2)=\sup \Big\{ m(\psi(z_1), \psi(z_2)):\psi\in \Psi\Big\}.
	\end{align}
	Moreover, if $\Psi$ is closed under pointwise convergence, then $d_\Psi (z_1,z_2)$ is attained by some element of $\Psi$.
\end{corollary}

\begin{proof} 
	We find a sequence $\{\psi_n\}$ in $\Psi$ such that $\{m(\psi_n(z_1),\psi_n(z_2))\}$ converges to $\sup \Big\{ m(\psi(z_1), \psi(z_2)):\psi\in \Psi\Big\}$. Since each element of $\Psi$ vanishes at a point of $\Omega$ and $\Psi$ separates the poins of $\Omega$, by Montel's theorem, there is a non-constant function $\tilde{\psi}:\Omega \rightarrow\mathbb{D}$ such that $\tilde{\psi}\in \SA_\Psi$ and $$\sup \Big\{ m(\psi(z_1), \psi(z_2)):\psi\in \Psi\Big\}=m(\tilde{\psi}(z_1), \tilde{\psi}(z_1)).$$
	Note that $m(\tilde{\psi}(z_1), \tilde{\psi}(z_1))\geq m(\psi(z_1), \psi(z_1))$ for all $\psi\in \Psi$.
	We now construct a function like $\tilde{k}$ as in (\ref{kernel_associated_to_function}) using $\tilde{\psi}$, that is,
		\begin{align*}
		\tilde{k}(x,y)&=\frac{1}{1- \tilde{\psi}(x)\overline{\tilde{\psi}(y)}},\,x,y\in\{z_1,z_2\}\\
		&=0,\,\,\text{otherwise}.\notag
	\end{align*} 
	
	Then arguing as in the proof of Theorem \ref{d_in_terms_of_adm_kernels}, we see that $\tilde{k}$ is an admissible kernel. Now we have
	\begin{align*}
		m(\tilde{\psi}(z_1), \tilde{\psi}(z_2)&=\sqrt{1-\frac{|\tilde{k}(z_1, z_2)|^2}{\tilde{k}(z_1,z_1)\tilde{k}(z_2,z_2)}}\\
		&\geq d_\Psi (z_1,z_2)\\
		&\geq m(\tilde{\psi}(z_1), \tilde{\psi}(z_2)),
	\end{align*}  
	and hence, (\ref{d_in_Psi}) follows.

	 When the collection $\Psi$ is closed in the topology of pointwise convergence, we have $\tilde{\psi}\in \Psi\subset \SA_\Psi$. This completes the proof.
\end{proof}

The next result shows that containing the Carath\'eodory extremals is necessary, not sufficient though (see page 47 in \cite{Agler1990}), for a collection of test functions to generate the Schur class.
\begin{corollary}\label{CaraAndTest}
	A necessary condition for a collection $\Psi$ of test functions to satisfy $\SA_\Psi=\mathscr{S}(\Omega)$ is that for any two points $z_1,z_2\in \Omega$, the Carath\'eodory pseudodistance $c_\Omega ^* (z_1,z_2)$ must be equal to $$\sup \Big\{ m(\psi(z_1), \psi(z_2)):\psi\in \Psi\Big\}.$$
\end{corollary}
\begin{proof}
	If $\SA_\Psi=\mathscr{S}(\Omega)$, then we have $d_\Psi =c^* _\Omega$, and hence, the result follows.
\end{proof}

Next we prove existence of a collection of test functions so that the structure in Theorem \ref{CharacterSchurAgler} holds for the Schur class $\mathscr{\Omega}$.

\begin{proposition}\label{UniversalTestFunctions}
	For any Carath\'eodory hyperbolic domain $\Omega\subset\mathbb{C}^m$, there is a collection $\Psi$ of test functions such that $\SA_\Psi=\mathscr{S}(\Omega)$.
\end{proposition}

\begin{proof}
	Let $w_0\in \Omega$ and let $\Psi =\{\psi_f (z)=\frac{f(w_0)-f(z)}{1-\overline{f(w_0)}f(z)}:f \in \mathscr{S} (\Omega)\}$. Then clearly, this $\Psi$ satisfies the second defining condition for a set of test functions, as it separates the points of $\Omega$. Also, for any point $z\in \Omega$
	\begin{align*}
		|\psi_f(z)|=\bigg|\frac{f (w_0)-f (z)}{1- \overline{f(w_0)}f(z)}\bigg|\leq c^* _\Omega  (w_0 , z)<1.
	\end{align*}
	Since $\Psi\subset\SA_\Psi$ and, for any  $a\in \mathbb{D}$ and $\psi \in \Psi$, $(a-\psi)(1-\overline{a}\psi)^{-1}\in \SA_\Psi$ we have $\mathscr{S}(\Omega)\subset\SA_\Psi$. This completes the proof.
\end{proof}
\begin{corollary}
	With the collection $\Psi$ of test functions described in the proof of Proposition \ref{UniversalTestFunctions}, the element $E(z)\in \CB$ has norm $c^*(w_0,z)$ for all $z\in \Omega$.
\end{corollary}
\begin{proof}
	Since, $c^*(w_0,z)$ is always attained by a function in $\mathscr{S}(\Omega)$, the result is clear from our choice of $\Psi$ and the fact that $||E(z)||=\text{sup}_{\psi\in \Psi} |\psi(z)|$.
\end{proof}
We now focus on our main result of this section. It says that under certain conditions, the disc and the bidisc are the only two domains whose Schur class can be generated by finitely many test functions. First we need a result which describes the minimal universal set for the Carath\'eodory extremal problem on $\D^n$.

\begin{lemma}\label{Uniqueness_Cara_Ext_Poly}
	Suppose that $F_1,\ldots,F_n$ are $n$ functions in $\mathscr{S}(\D^n)$ such that 
	\begin{align*}
		c^* _{\D^n}(u_1,u_2)=max\{m(F_j (u_1),F_j(u_2)):1\leq j\leq n\},\,\,\,\text{for all}\,\,\,u_1,u_2\in \D^n.
	\end{align*}
	Then the set $\{F_1,\ldots,F_n\}$ and the set of the coordinate functions $\{z_1,\ldots,z_n\}$ are equal upto composition with automorphisms of $\D$.
\end{lemma}
\begin{proof}
	We first note that $\{F_1,\ldots,F_n\}$ is a universal set for the Carath\'eodory extremal problem on $\D^n$. We define $\Phi:\D^n\rightarrow \D^n$ by 
	\begin{align*}
		\Phi(z)=(F_1(z),\ldots, F_n(z)),\,\,\,\ z\in \D^n.
	\end{align*}
	By Theorem 4.4 in \cite{KosinZwoExt}, $\Phi$ is a biholomorphism from $\D^n$ onto $\D^n$, that is, $\Phi$ is an automorphism of $\D^n$. So there are automorphisms $\phi_1,\ldots, \phi_n$ of $\D$ and a permutation $\sigma$ of $\{1,\ldots,n\}$ such that $\Phi(z)=\Phi(z_1,\ldots, z_n)=(\phi_1(z_{\sigma(1)}),\ldots,\phi_n (z_{\sigma(n)}))$ for all $z\in \D^n$ (see page 790 in \cite{J-P-Invariant}). This is sufficient to conclude the proof.
\end{proof}

Next we observe the following: Suppose that $\Psi=\{\psi_j:1\leq j\leq n\}$ is finite and for each $j$, we consider the function $\tau_j\in \CB$ satisfying $\tau_j(\psi_k)=\delta_{jk}$ (the Kronecker delta). Then, the collection $\{\tau_j:1\leq j\leq n\}$ generates $\CB$ and for any fixed $z\in \Omega$ we can write $E(z)=\sum_{j=1}^{n} \psi_j(z)\tau_j$. In this case, we conclude from (\ref{AglerDecompInner}) that for a given $f\in\SA_\Psi$ there is a Hilbert space $\X$, a unital $*$-representation $\rho:\CB\rightarrow B(\X)$ and a map $h:\Omega\rightarrow \X$ such that
\begin{align}\label{AglerDecomp}
	1-\overline{f(z)}f(w)=\sum_{j=1}^{n} (1-\overline{\psi_j(z)}\psi_j(w)) \langle\rho(\tau_j)h(w),h(z)\rangle_\X \,\,\,\text{for all $z,w\in \Omega$.}
\end{align}
To state the main result, we need the following two notions:
\begin{enumerate}
	\item The first one is \textit{$\gamma$-hyperbolicity}. We call a domain  $\gamma$-hyperbolic if its Carath\'eodory–Reiffen pseudometric is nondegenrate (\cite{J-P-Invariant}). 
	\item Secondly, we call a domain $\Omega$ \textit{$c^*$-finitely compact} if for each $z_1\in\Omega$ and $r\in (0,1)$ the set $\{z_2\in \Omega:c_\Omega ^*(z_1,z_2)<r\}$ is a relatively compact subset of $\Omega$ (\cite{KosinZwoExt}).
\end{enumerate}

\begin{theorem}
	Suppose that $\Omega$ is a domain which is Carath\'eodory hyperbolic, $\gamma$-hyperbolic and $c^*$-finitely compact. Then, the following are equivalent.
	\begin{enumerate}
		\item $\Omega\subset \mathbb{C}^m$ and there is a collection $\Psi=\{\psi_j:1\leq j\leq m\}$ consisting of $m$ functions such that $\SA_\Psi=\mathscr{S}(\Omega)$.
		\item $m=1$ or $2$ and, $\Omega$ is biholomorphic with $\D^m$.
	\end{enumerate}
\end{theorem}
\begin{proof}
	$(2)\implies (1):$ We prove it for the case $m=2$. Let $F:\Omega\rightarrow \D^2$ be a biholomorphism and let $z_1$ and $z_2$ denote the coordinate functions on $\D^2$. Then clearly, $\psi_j=z_j\circ F$ are test functions on $\Omega$. For an $f\in \mathscr{S}(\Omega)$, we find that $f\circ F^{-1}\in \mathscr{S}(\D^2)$, and hence, has an Agler decomposition of the form (\ref{AglerDecomp}) in terms of $z_1$ and $z_2$. From this we conclude that $f$ has an Agler decomposition in terms of $\psi_j,j=1,2$.
	
	$(1)\implies(2):$ When $\SA_\Psi=\mathscr{S}(\Omega)$, by Corollary \ref{CaraAndTest}, we see that $\Psi$ is a universal set for the Carath\'eodory extremal problem. By Theorem 4.4 in \cite{KosinZwoExt}, the map $\Phi:\Omega\rightarrow \D^m$ sending $w$ to $(\psi_1(w)\ldots,\psi_m (w))$ is a biholomorphism. Invariance of the Carath\'eodory pseudodistance says that $\{\psi_j\circ \Phi^{-1}:1\leq j\leq m\}$ is a universal set for the Carath\'eodory extremal problem on $\D^m$. By Lemma \ref{Uniqueness_Cara_Ext_Poly}, we see that the set $\{\psi_j\circ \Phi^{-1}:1\leq j\leq m\}$ is just the set $\{z_j:1\leq j\leq m\}$ of $m$ coordinate functions. If $f\in \mathscr{S}(\D^m)$ is any function, then $f\circ \Phi\in \SA_\Psi$, and hence, has an Agler decomposition of the form (\ref{AglerDecomp}) in terms of the functions in $\Psi$. Consequently, $f$ has an Agler decomposition in terms of the coordinate functions $z_j$. Since $f$ is arbitrary, this is a contradiction unless $m=1$ or $2$ (see page 47 in \cite{Agler1990}).
\end{proof}
\section{Applications}\label{Applications}
\subsection{Carath\'eodory extremals and the multiplier algebra of the Drury-Arveson space}
Following the description in \cite{A-M}, we consider the kernel $K_m:\mathbb{B}_m\times \mathbb{B}_m\rightarrow \mathbb{C}$ defined by
\begin{align*}
	K_m(z,w)=\frac{1}{1-\langle z,w\rangle_m},
\end{align*}
where $\mathbb{B}_m$ is the Euclidean unit ball in $\mathbb{C}^m$ and $\langle .,.\rangle_m$ is the usual inner product in $\mathbb{C}^m$.
The reproducing kernel Hilbert space generated by $K_m$ is known as the \textit{Drury–Arveson space}.

We observe that for any $z_1,z_2\in \mathbb{B}_m$ 
\begin{align*}
	\sqrt{1-\frac{|K_m(z_1, z_2)|^2}{K_m(z_1,z_1)K_m(z_2,z_2)}}=c^* _{\mathbb{B}_m} (z_1,z_2).
\end{align*}
So, the interpolation problem $z_1\mapsto 0$ and $z_2\mapsto c^* _{\mathbb{B}_m} (z_1,z_2)$ has a solution in the closed unit ball of the multiplier algebra of the Drury-Arveson space, that is, the closed unit ball of the multiplier algebra contains a Carath\'eodory extremal. We will describe the possible values of this Carath\'eodory extremal.
\begin{theorem}
	Suppose that $z_1,z_2\in \mathbb{B}_m$ and $f:\mathbb{B}_m\rightarrow \mathbb{D}$ is a solution to the interpolation problem $z_1\mapsto 0$ and $z_2\mapsto c^* _{\mathbb{B}_m} (z_1,z_2)$ in the closed unit ball of the multiplier algebra of the Drury-Arveson space. Then, for any $z_3\in \mathbb{B}_m$ we have
	\begin{align*}
		f(z_3)\in \Big\{w\in \D : \frac{1-c^* _{\mathbb{B}_m} (z_1,z_2) \overline{w}}{|1-c^* _{\mathbb{B}_m} (z_1,z_2) \overline{w}|}=\frac{1-\langle z_2,z_3\rangle_m}{|1-\langle z_2,z_3\rangle_m|}\Big\}.
	\end{align*}
	In particular, $f$ is always real valued on the set $\{z\in \mathbb{B}_m:\langle z,z_2\rangle =0\}$.
\end{theorem}
\begin{proof}
	Let $z_3\in \mathbb{B}_m$ and $f(z_i)=w_j,j=1,2,3$. We write $$k_{i,j}=\frac{K_m (z_i,z_j)}{\sqrt{K_m (z_i,z_i)K_m (z_j,z_j)}}$$
	and $m_{i,j}=m(w_i,w_j)$.
	
	Since the matrix $\big((1-w_i\overline{w_j})K_m (z_i,z_j)\big)$ is positive semidefinite, so is the matrix $\big((1-w_i\overline{w_j})k_{i,j}\big)$. Taking the Schur product of this matrix with
	\begin{align*}
		\begin{pmatrix}
			1&e^{i\theta}&e^{i\sigma}\\
			e^{-i\theta}&1&e^{i(\sigma-\theta)}\\
			e^{-i\sigma}&e^{i(-\sigma+\theta)}&1
		\end{pmatrix},
	\end{align*} 
	we may assume that $k_{1,2}$ and $k_{1,3}$ are positive. Now, using the fact that $\big((1-w_i\overline{w_j})k_{i,j}\big)$ has non-negative determinant, we find
	\begin{align}\label{Large_ineq}
		\frac{2k_{1,2} |k_{2,3}|k_{3,1}Re\Biggl\{\frac{k_{2,3}}{|k_{2,3}|}\frac{(1-w_1 \overline{w_2})(1-w_2 \overline{w_3})(1-w_3\overline{w_1})}{|(1-w_1 \overline{w_2})(1-w_2 \overline{w_3})(1-w_3\overline{w_1})|}\Biggr\} }{\sqrt{(1-m_{1,2}^2)(1-m_{2,3}^2)(1-m_{3,1}^2)}}\\
		\geq\frac{k_{1,2}^2}{(1-m_{1,2}^2)}+\frac{|k_{2,3}|^2}{(1-m_{2,3}^2)}+\frac{k_{3,1}^2}{(1-m_{3,1}^2)}-1.\notag
	\end{align}
	We have $k_{1,2}=\sqrt{1-c^* _{\mathbb{B}_m}(z_1,z_2)^2}=\sqrt{1-m_{1,2} ^2}$, $w_1=0$ and $w_2=c^* _{\mathbb{B}_m} (z_1,z_2)$. So we get
	\begin{align*}
		&\frac{k_{1,2}^2}{(1-m_{1,2}^2)}+\frac{|k_{2,3}|^2}{(1-m_{2,3}^2)}+\frac{k_{3,1}^2}{(1-m_{3,1}^2)}-1\\
		&=\frac{|k_{2,3}|^2}{(1-m_{2,3}^2)}+\frac{k_{3,1}^2}{(1-m_{3,1}^2)}\\
		&\geq 2\frac{|k_{2,3}|k_{3,1}}{\sqrt{(1-m_{2,3}^2)(1-m_{3,1}^2)}},
	\end{align*}
	 and we conclude from (\ref{Large_ineq}) that $Re\Big\{\frac{k_{2,3}}{|k_{2,3}|}\frac{1-c^* _{\mathbb{B}_m} (z_1,z_2) \overline{w}}{|1-c^* _{\mathbb{B}_m} (z_1,z_2) \overline{w}|}\Big\}\geq 1$. From this, it is easy to see that $w_3$ must satisfy \begin{align*}
	\frac{1-c^* _{\mathbb{B}_m} (z_1,z_2) \overline{w_3}}{|1-c^* _{\mathbb{B}_m} (z_1,z_2) \overline{w_3}|}=\frac{1-\langle z_2,z_3\rangle_m}{|1-\langle z_2,z_3\rangle_m|}.
	\end{align*}
\end{proof}

\subsection{A Herglotz Representation}

For a domain $\Omega$, the \textit{Herglotz class} on $\Omega$ is the class of all holomorphic functions $f:\Omega\rightarrow \mathbb{C}$ that satisfy $Re(f(z))\geq 0$ for all $z\in \Omega$. The classical Herglotz representation theorem (\cite{Herglotz}) states that for any such function $f$ on the $\D$, we can find a real number $b$ and a finite
non-negative Borel measure $\mu$ on the circle $\partial\D$ such that

$$f(z)=ib+\int_{\partial\D} \frac{1+tz}{1-tz}d\mu(t),\,\,\text{for all}\,\,z\in \D.$$

There are a few generalizations of this result, and among them, our focus will be on the operator-theoretic ones.  Our argument closely follows that in the proof of Theorem 6.12 in \cite{AglerMcCYouBook}.

\begin{theorem}
	Let $\Omega\subset\mathbb{C}^m$ be a Carath\'eodory hyperbolic domain, fix a point $w_0\in\Omega$, and define the family of test functions
	\[
	\Psi \;=\; \Bigl\{\psi_f(z)=\frac{f(w_0)-f(z)}{1-\overline{f(w_0)}f(z)}\;:\; f\in\mathscr{S}(\Omega)\Bigr\},
	\]
	Let $\mathcal{C}_b(\Psi)$ be the $C^*$-algebra of bounded continuous functions on $\Psi$ and $E(z)\in\mathcal{C}_b(\Psi)$ the evaluation map $E(z)(\psi)=\psi(z)$ for $z\in\Omega$ and $\psi\in \Psi$.
	
	Then, for any holomorphic function $h:\Omega\rightarrow \mathbb{C}$, the following are equivalent.
	\begin{enumerate}
		\item $Re(h(z))\geq 0$ for all $z\in \Omega$.
		\item There exist a real number $b$, a Hilbert space $\X$, a unitary operator $U\in B(\X)$, a $\gamma\in \X$ and a unital $*$-representation $\rho:\CB\rightarrow B(\X)$ such that one can write
		\begin{align}\label{FormOfh}
			h(z)=ib+\Big\langle\gamma,(1+U\rho(E(z)^*))(1-U\rho(E(z)^*))^{-1}\gamma \Big\rangle\,\,\,\text{for all}\,\,\,z\in \Omega.
		\end{align}
	\end{enumerate}
\end{theorem}
\begin{proof}
	It easily follows that $(2)$ implies $(1)$. So we suppose that $(1)$ holds. We choose $w_0$ and $\Psi$ as in the proof of Proposition \ref{UniversalTestFunctions}. First, we assume that $h(w_0)=1$. Now with our choice of $\Psi$, we have $\frac{h-1}{h+1}\in \SA_\Psi$ and by Theorem \ref{CharacterSchurAgler}, there is a Hilbert space $\X$, a unital $*$-representation $\rho:\CB\rightarrow B(\X)$ and a map $g:\Omega\rightarrow \X$ such that 
	\begin{align}\label{AglerDecompForRatio}
		1-\frac{h(z)-1}{h(z)+1}\frac{\overline{h(w)}-1}{\overline{h(w)}+1}=\Big\langle\rho(1-E(z)E(w)^*)g(w),g(z)\Big\rangle_\X,\,\,\text{for all}\,\,z,w\in \Omega.
	\end{align}
	We write $k(z)=(\overline{h(z)}+1)g(z)$ and deduce from (\ref{AglerDecompForRatio}) that
	\begin{align}\label{InnerProductForm}
		\Bigg\langle \begin{pmatrix}  \overline{h(w)}-1\\ k(w) \end{pmatrix}, \begin{pmatrix} \overline{h(z)}-1\\ k(z)  \end{pmatrix} &\Bigg\rangle_{ \mathbb{C}\oplus \X}=\\
		&\Bigg\langle \begin{pmatrix}  \overline{h(w)}+1\\ \rho(E(w)^*)k(w) \end{pmatrix}, \begin{pmatrix} \overline{h(z)}+1\\ \rho(E(z)^*)k(z)  \end{pmatrix} \Bigg\rangle_{ \mathbb{C}\oplus \X}.\notag
	\end{align}
	Therefore, we can construct a unitary operator $V:\mathbb{C}\oplus\X\rightarrow \mathbb{C}\oplus \X$ such that $V$ can be written as
	\begin{align*}
		V=
		\bordermatrix{ & \mathbb{C} & \X \cr
			\mathbb{C} & a & B \cr
			\X & C & D},
	\end{align*}
	and it sends $\begin{pmatrix} \overline{h(z)}+1\\ \rho(E(z)^*)k(z)  \end{pmatrix}$ to $\begin{pmatrix} \overline{h(z)}-1\\ k(z)  \end{pmatrix}$ for all $z\in \Omega$. Since $B\in B(\X,\mathbb{C})$ and $C\in B(\mathbb{C},\X)$, there exist unique $\beta$ and $\gamma$ in $\X$ such that $B=1\otimes \beta$ and $C=\gamma\otimes 1$ where we follow the convention $(u\otimes v)(t)=\langle t,v\rangle u$. 
	
	Since $h(w_0)=1$ and $E(w_0)=0$, we deduce from $$a(\overline{h(z)}+1)+\langle \rho(E(z)^*)k(z),\beta\rangle=\overline{h(z)}-1$$ that $a=0$. Using the fact that $V$ is a unitary we also find that $||\gamma||=1$ and the operator $U=\gamma\otimes\beta+D$ is unitary on $\X$.
	
	In terms of the components of $V$, we now have $$(\overline{h(z)}+1)\gamma + D\rho(E(z)^*)k(z)=k(z)$$
	and $\langle\rho(E(z)^*)k(z),\beta\rangle=\overline{h(z)}-1$, and we deduce
	\begin{align}\label{ValuesOfk}
		k(z)&=U\rho(E(z)^*)k(z) +2\gamma=2(1-U\rho(E(z)^*))^{-1}\gamma.
	\end{align}
	Putting $w=w_0$ in (\ref{InnerProductForm}), we find $2(1+h(z))=\langle 2\gamma, k(z)\rangle$ and then using the formulas for $k(z)$ from (\ref{ValuesOfk}) one can easily deduce the formula (\ref{FormOfh}) with $b=0$.
	
	For the general case we take $h(w_0)=a+ib$ with real $a$ and $b$. If $a=0$, then open mapping theorem says that $h$ is constant and the formula (\ref{FormOfh}) then holds trivially. If $a\neq 0$, then necessarily $a>0$. We replace $h$ by $g=\frac{h-ib}{a}$ to obtain the desired representation. This completes the proof.
\end{proof}

\subsection*{Acknowledgements}
This work was partially supported by GAČR, Czech Science Foundation, grant 22-15012J, which was active during the initial stages of the research. The completion of this work was supported by European Horizon MSCA grant, CZ.02.01.01/00/22 10/0008854. The author gratefully acknowledges the financial assistance provided by both grants.

%%%%%%%%%%% To ease editing, use normal size for the references:


\normalsize


\begin{thebibliography}{[HD82]}

%% Use the widest label as parameter above.
%% Reference items can be numbered or have labels of your choice, as below.
%% Arrange the items in the alphabetical order of names (and not in the order of labels).

%% In IMPAN journals, only the title is italicized; boldface is not used.
%% Do NOT give the issue number unless the issues are paginated separately, as in Uspekhi below.

%% To ease editing, add:

\normalsize
\baselineskip=17pt


%%%%%%%%%%%%%

	\bibitem{Agler1990} Jim Agler. On the representation of certain holomorphic functions defined on a polydisc. In Topics in operator theory: Ernst D. Hellinger memorial volume, volume 48 of Oper. Theory
Adv. Appl., pages $47-66$. Birkh\"auser, Basel, 1990

\bibitem{AglLykYouCharac} J. Agler, Z. Lykova and N. J. Young, \textit{Characterizations of some domains via Carathéodory extremals}, J. Geom. Anal. 29 (2019), no. 4, $3039-3054$.


\bibitem{A-M-Paper} J. Agler and J.E. McCarthy, \textit{Nevanlinna-Pick interpolation on the bidisk}, J. Reine Angew. Math. \textbf{506}, 1999, pp. $191-204$.

\bibitem{A-M} J. Agler and J.E. McCarthy, \textit{Pick Interpolation and Hilbert Function Spaces,} Graduate Studies in Mathematics Vol. 44, American Mathematical Society, Providence, 2002.

\bibitem{AglerMcCYouBook} J. Agler, J. E. McCarthy, and N. J. Young, \textit{Operator analysis—Hilbert space methods in complex analysis}, Cambridge Tracts in Math., 219, Cambridge University Press, Cambridge, 2020.

\bibitem{AglerYounHyperbolic} J. Agler and N. J. Young, \textit{The hyperbolic geometry of the symmetrized bidisc}, J. Geom. Anal. 14 (2004), $375-403$.


\bibitem{B-B-F-t} J. A. Ball, A. Biswas, Q. Fang and S. ter Horst, \textit{Multivariable generalizations of the Schur class: positive kernel characterization and transfer function realization,} Recent Advances in Operator Theory and Applications, OT 187 Birkh\"auser-Verlag, Basel, 2008, pp. $17-79$.

\bibitem{B_H} J. A. Ball and M. D. Guerra Huam\'an, \textit{Test functions, Schur-Agler classes and transfer-function realizations: the matrix-valued setting,} Complex Anal. Oper. Theory 7 (2013), pp. $529-575$.


\bibitem{B-B-L-S} S. D. Barreto, B. V. R. Bhat, V. Liebscher and M. Skeide, \textit{Type I product systems of Hilbert modules}, J. Func. Anal. \textbf{212} (2004), pp. $121-181$.


\bibitem{BhatBisCha} T. Bhattacharyya, A. Biswas, V. Singh Chandel, \textit{On the Nevanlinna problem: characterization of all Schur-Agler class solutions affiliated with a given kernel,} Studia Math. 255 (2020), no. 1, $83-107$.



\bibitem{B-S} T. Bhattacharyya and H. Sau, \textit{Holomorphic functions on the symmetrized bidisk - realization, interpolation and extension,}  J. Funct. Anal. \textbf{274} (2018), pp. $504-524$.



\bibitem{D-M} M. A. Dritschel and S. McCullough, \textit{Test functions, kernels, realizations and interpolation,} in: Operator Theory, Structured Matrices, and Dilations. Tiberiu Constantinescu Memorial Volume (ed. M. Bakonyi, A. Gheondea, M. Putinar and J. Rovnyak), Theta Foundation, Bucharest, 2007, pp. $153 -179$.

\bibitem{Herglotz} G. Herglotz,	\textit{\"Uber Potenzreihen mit positivem, reellem Teil in Einheitskreis}, Leipziger Berichte, Mathematics, Physics, Vol. 63, 1911, $501-511$.



\bibitem{J-P-Invariant} M. Jarnicki and P. Pflug, \textit{Invariant Distances and Metrics in Complex Analysis,} 2nd extended edition, De Gruyter, Berlin, 2013.

\bibitem{KosinZwoExt} Ł. Kosiński and W. Zwonek, \textit{Extension property and universal sets}, Canad. J. Math. 73 (2021), no. 3, $717-736$.

% \bibitem{Nagy-Foias} B. Sz.-Nagy, C. Foias, H. Bercovici, L. K\'erchy, \textit{Harmonic analysis of operators on Hilbert space.} Second edition. Revised and enlarged edition. Universitext. Springer, New York, 2010.


\end{thebibliography}
\end{document}